\documentclass{amsart}[12 pts]

\usepackage{amsthm}
\usepackage[english]{babel}
\usepackage{amsmath}
\usepackage{amssymb}
\usepackage[%
]{hyperref}
\usepackage{booktabs}
\usepackage{enumerate}
\usepackage{multicol}
\usepackage{bbm}
\usepackage[]{todonotes}
\usepackage{setspace}
\usepackage{mathtools}

\theoremstyle{definition}
\newtheorem{definition}{Definition}
\newtheorem{Remark}[definition]{Remark}

\theoremstyle{plain}
\newtheorem{theorem}[definition]{Theorem} 
\newtheorem{theoremintro}{Theorem}
\newtheorem{introcor}[theoremintro]{Corollary}
\newtheorem{lemma}[definition]{Lemma}

\newtheorem{corollary}[definition]{Corollary}

\DeclareMathOperator{\GL}{GL}
\DeclareMathOperator{\SL}{SL}
\DeclareMathOperator{\SU}{SU}

\DeclareMathOperator{\SP}{SP}
\DeclareMathOperator{\Aut}{Aut}

\DeclareMathOperator{\Irr}{Irr}
\DeclareMathOperator{\ord}{ord}

\newcommand{\N}{\mathbb{N}}
\newcommand{\F}{\mathbb{F}}
\newcommand{\Z}{\mathbb{Z}}

\numberwithin{definition}{section}
\numberwithin{equation}{section}

\title[Semi-extraspecial $p$-groups]{Semi-Extraspecial $p$-groups with automorphisms of large order}
\author[Brenner]{Sofia Brenner}
\address{Department of Mathematics, TU Darmstadt, 64289 Darmstadt, Germany.}
\email{sofia.brenner@uni-kassel.de}

\author[Camina]{Rachel D. Camina}
\address{Fitzwilliam College,  Cambridge, CB3 0DG, UK.} 
\email{rdc26@cam.ac.uk}

\author[Lewis]{Mark L. Lewis}
\address{Department of Mathematical Sciences, Kent State University, Kent, OH 44242, USA.}
\email{lewis@math.kent.edu}

\subjclass[2020]{20D15}
\keywords{Semi-extraspecial $p$-groups, automorphisms, special unitary group}

\begin{document}

\begin{abstract}
In this paper, we consider semi-extraspecial $p$-groups $G$ that have an automorphism of order $|G:G'| - 1$.  We prove that these groups are isomorphic to Sylow $p$-subgroups of $\SU_3 (p^{2a})$ for some integer $a$.  If $p$ is odd, this is equivalent to saying that $G$ is isomorphic to a Sylow $p$-subgroup of $\SL_3 (p^a)$.
\end{abstract}

\maketitle

\section{Introduction}

All groups in this paper are finite.  A $p$-group~$G$ is called \emph{semi-extraspecial} if $G/N$ is extraspecial for every maximal subgroup~$N$ of $Z(G)$.  Semi-extraspecial $p$-groups and, more generally, Camina $p$-groups have been a focus of extensive research (for instance, see \cite{BEI77, DaSc, LEW18, LEW19, Mac, Ver}). We present an overview in Section~\ref{background}. 
\medskip

In this paper, we study semi-extraspecial $p$-groups $G$ that possess an automorphism of order $|G:G'|-1$, where $G'$ denotes the derived subgroup as customary. Our motivation is two-fold. On the one hand, we show that this case is extremal in the sense that $|G:G'|-1$ is the maximal possible order of a $p'$-automorphism of a semi-extraspecial group $G$. On the other hand, semi-extraspecial $p$-groups with 
an automorphism of this order play a fundamental role in the first author's work~\cite{brenner,brku} on group algebras of finite groups in which the socle of the center is an ideal: essentially, groups with this property can be built as central products of extensions of semi-extraspecial $p$-groups $G$ by automorphisms of order $|G:G'|-1$ (see \cite[Theorem~B]{brenner}).

In our main theorem, we classify the semi-extraspecial $p$-groups that possess such an automorphism. Our main result is the following: 

\begin{theoremintro} \label{Intro}
Let $p$ be a prime number and let $G$ be a semi-extraspecial $p$-group. Then $G$ possesses an automorphism of order $|G:G'|-1$ if and only if $G$ is isomorphic to a Sylow $p$-subgroup of $\SU_3 (p^{2a})$ where $|G:G'| = p^{2a}$.
\end{theoremintro}

Here, $\SU_3(p^{2a})$ denotes the 3-dimensional special unitary group over the field with $p^{2a}$ elements. The structure of its Sylow $p$-subgroups is well-understood.  When~$p$ is odd, they are isomorphic to the Sylow $p$-subgroups of $\SL_3 (p^a)$, which are commonly known as {\it Heisenberg groups}. We thus obtain the following immediate corollary: 

\begin{introcor} \label{Intro odd}
Let $p$ be an odd prime number and let $G$ be a semi-extraspecial $p$-group. Then $G$ possesses an automorphism of order $|G:G'|-1$ if and only if $G$ is isomorphic to a Sylow $p$-subgroup of $\SL_3 (p^a)$ where $|G:G'| = p^{2a}$.	
\end{introcor}

In contrast, we note that the Sylow $2$-subgroups of $\SU_3 (2^{2a})$ are not isomorphic to the Sylow $2$-subgroups of $\SL_3 (2^a)$.
 \medskip
 
The paper is organized as follows: in Section~\ref{background}, we summarize some well-known facts on semi-extraspecial $p$-groups and their automorphisms. In Section~\ref{sec:proofmainresults}, we then prove our main results. 

\section{Background}\label{background}

Let~$G$ be a group. As is customary, we write $G'$ for its derived subgroup and $Z(G)$ for its center, we set $\Phi(G)$ for its Frattini subgroup, we let $\Aut(G)$ denote the automorphism group of~$G$, and we write $\Irr(G)$ for the set of irreducible characters of~$G$.  Denote $\ord(g)$ for the order of an element $g \in G$, and $\max(G)$ for the maximal order of an element in $G$. For a prime number~$p$, a \emph{$p'$-automorphism} of $G$ is an automorphism of order not divisible by~$p$. 

Let $G$ be a $p$-group for some prime number $p$.  We set $\Omega (G) \coloneqq \langle x \colon x^p = 1 \rangle$ and $\mho (G) \coloneqq \langle x^p \colon x \in G \rangle$. 
Recall that~$G$ is called \emph{special} if either~$G$ is elementary abelian, or if $G$ has class $2$ and $G' = Z(G) = \Phi(G)$ holds. The group~$G$ is called \emph{extraspecial} if $|G'| = |Z(G)| = p$ holds (and hence $G$ is special). 

For $n, k \in \N$, a prime number $p$ and $q \coloneqq p^k$, we write 
$\GL_n (q)$ for the {\it general linear group} of $n \times n$-matrices with entries in $\F_q$ whose determinant is not $0$. It is well-known that $\max(\GL_n(q)) = q^n-1$. Elements of this order are called \emph{Singer cycles}. Let $\SL_n (q)$ denote the {\it special linear group} of dimension $n$ over $\F_q$.  This is the subgroup of $\GL_n(q)$ consisting of elements of determinant~$1$.   

Finally, we write $\SU_n(q^2)$ for the {\it special unitary group} of dimension~$n$ over the field $\F_{q^2}$.  We see that $\SU_n (q^2)$ is the subgroup of unitary matrices of determinant~$1$ in $\GL_n(q^2)$. Recall that a matrix $U$ is {\it unitary} if $U^{-1} = \overline U^t$ where $\overline {U}$ means to take the $q^{\rm th}$ powers of each entry of $U$.  (See \cite{atlasart} or \cite[Section II.10]{HUP67}).

It is well-known that $\GL_k (p)$ can be thought of as the automorphism group of a vector space $V$ of dimension $k$ over the field of $p$ elements.  When $V$ admits an alternating non-degenerate bilinear form, then the dimension of $V$ must be even, say $2k$, and the subgroup of $\GL_{2k} (p)$ that preserves this form is the {\it symplectic group} $\SP_{2k} (p)$. (Again see \cite{atlasart}).

A finite group $G$ is called a \emph{Camina group} if for all $g \in G \setminus G'$, we have that~$g$ is conjugate to all elements in $gG'$. This is one of several equivalent conditions that can be used to define Camina groups (see \cite[Lemma~2.2]{LEW18}). Dark and Scoppola have proved that a Camina group is either a Frobenius group whose Frobenius complement is cyclic or the quaternion group, or a $p$-group (see \cite[Main theorem]{DaSc}). 

The $p$-groups that are Camina groups have a very special structure.  MacDonald has shown that Camina $2$-groups have nilpotence class $2$ (see \cite[Theorem~3.1]{Mac}). Dark and Scoppola have proved in \cite{DaSc} that the nilpotence class of a Camina $p$-group for odd $p$ is~$2$ or~$3$.  For Camina $p$-groups of nilpotence class $2$, there is an equivalent characterization.  A $p$-group~$G$ is called \emph{semi-extraspecial} (S.E.S.) if $G/N$ is extraspecial for every subgroup $N$ that is a maximal subgroup of $Z(G)$.  Verardi has proved in \cite{Ver} that~$G$ is a Camina $p$-group of nilpotence class $2$ if and only if~$G$ is a semi-extraspecial $p$-group.  The third author has written an expository article \cite{LEW18} that covers many of the known results including the ones mentioned in this paper about Camina groups and related objects.

Since the groups of interest in this paper are Camina groups of nilpotence class~$2$, they are semi-extraspecial $p$-groups, and we will refer to them as such for the remainder of this paper.  The third author has also written a paper with an extensive expository account  on the known results on semi-extraspecial $p$-groups in~ \cite{LEW19}. We collect well-known facts on semi-extraspecial $p$-groups: 

\begin{lemma}[{\cite[Lemma 1 and Satz 1]{BEI77} and \cite[Section~2]{brown}}]\label{lemma:ses}
Let $p$ be a prime and let $G$ be a semi-extraspecial $p$-group. Then the following hold: 
\begin{enumerate}
\item We have $G' = Z(G) = \Phi(G)$, so $G/Z(G)$ elementary abelian. 
\item We have $|G: Z(G)| = p^{2a}$ for some integer $a$. 
\item The group $Z(G)$ is elementary abelian. Writing $|Z(G)| = p^b$, we have $b \leq a$, where $a$ is as in (2).
\item For a nontrivial character $\phi \in \Irr(G)$, setting $\langle x Z(G), y Z(G) \rangle_\phi = \phi([x,y])$ for all $x,y \in G$ defines an alternating bilinear form on $G/Z(G)$. 
\end{enumerate}
\end{lemma}

Using the notation of Lemma~\ref{lemma:ses}, a semi-extraspecial $p$-group $G$ is called \emph{ultraspecial} when $a = b$. 

\section{Proofs of the Main Results}\label{sec:proofmainresults}

Throughout, let $p$ be a prime number. In this section, we study the structure of semi-extraspecial $p$-groups with an automorphism of large order in detail and prove Theorem~\ref{Intro}. 

We begin by showing that $|G:G'|-1$ is the largest possible order of a $p'$-automorphism of a semi-extraspecial group $G$ and characterize the actions of automorphisms of this order. To this end, we first prove the following more general statement:

\begin{lemma}\label{lemma:gen transitive}
Let $G$ be a $p$-group with $|G:\Phi (G)| = p^{k}$ and let $\varphi$ be a $p'$-auto\-mor\-phism of $G$. Then $\ord(\varphi) \leq p^{k}-1$. Equality holds if and only if $\varphi$ permutes the nontrivial cosets of $\Phi (G)$ in $G$ transitively.
\end{lemma}

\begin{proof}
Note that $\varphi$ induces an automorphism~$\varphi'$ of the elementary abelian group $G/\Phi (G)$.  Let $d \coloneqq \ord(\varphi')$, so $\varphi^d$ acts trivially on $G/\Phi (G)$. By Burnside's theorem (see \cite[Theorem 5.1.4]{GOR68}), we have $\varphi^d = 1$ and hence, $d = \ord(\varphi)$ follows.  The automorphism group of $G/ \Phi (G)$ is isomorphic to $\GL_{k}(p)$.  Hence, we obtain $\ord(\varphi) \leq \max(\GL_{k}(p)) = p^{k}-1$. Moreover, equality holds if and only if $\varphi'$ is a Singer cycle, and therefore, permutes the nontrivial elements in $G/\Phi (G)$ transitively. 
\end{proof}

We apply this lemma to  S.E.S. groups.

\begin{corollary}\label{lemma:transitive}
Let $G$ be a semi-extraspecial $p$-group with $|G:Z(G)| = p^{2a}$. For every $p'$-automorphism $\varphi$ of $G$, we have $\ord(\varphi) \leq p^{2a}-1$. Moreover $\ord(\varphi) = p^{2a}-1$ if and only if $\varphi$ permutes the nontrivial cosets of $Z(G)$ in~$G$ transitively. 
\end{corollary}

\begin{proof}
When $G$ is an S.E.S. group, we have that $Z(G) = \Phi (G)$.  Then we can apply Lemma \ref{lemma:gen transitive} to obtain the conclusion.
\end{proof}

To illustrate the strength of this condition, we show that if $p$ is odd, the existence of an automorphism that permutes the nontrivial cosets of $Z(G)$ transitively forces the semi-extraspecial group to be of exponent $p$. Note that the analogous statement for $p = 2$ cannot hold as groups of exponent~2 are abelian. 

\begin{lemma}\label{exponent}
Assume $p > 2$ and let $G$ be a semi-extraspecial $p$-group. If $G$ has an automorphism that acts transitively on the nontrivial cosets of $G/Z(G)$, then~$G$ is of exponent~$p$. 
\end{lemma}

\begin{proof}
Since $G$ has nilpotence class $2$, $G/Z(G)$ and $Z(G)$ have exponent $p$ and $p$ is odd, the map $G/Z(G) \to Z(G)$ defined by $xZ(G) \rightarrow x^p$ is a homomorphism with kernel $\Omega /Z(G)$ and image $\mho$ (see \cite[Hilfssatz III.1.3]{HUP67}). By the First Isomorphism theorem, $G/\Omega$ is isomorphic to $\mho$, so $|G:\Omega| = |\mho|$.  As normal, set $|G:Z(G)| = p^{2a}$ and $|Z(G)| = p^b$.  If $G$ has exponent $p^2$, then $1 < |\mho| \leq p^b$. This implies that $1 < |G/\Omega| \leq p^b$.  Hence, $|\Omega/Z(G)| \ge p^{2a - b} \geq p^a > 1$.  It follows that $G' = Z(G) < \Omega < G$.  On the other hand, $\Omega$ is characteristic in $G$. This contradicts $G$ having an automorphism that acts transitively on the nontrivial elements in $G/Z(G) = G/G'$.  Therefore, $G$ must have exponent $p$.
\end{proof}

Now we turn to the proof of Theorem~\ref{Intro}. To this end, we first refine the results on orders of $p'$-automorphisms from Corollary~\ref{lemma:transitive}. 
As is standard notation, when~$\delta$ is an automorphism of $G$, then we write $\delta_{Z(G)}$ for the restriction of $\delta$ to $Z(G)$.

\begin{lemma}\label{max auto}
Let $G$ be a semi-extraspecial $p$-group with $|G:Z(G)| = p^{2a}$ and $|Z(G)| = p^b$. Let $\delta$ be a $p'$-automorphism of $G$ and set $d \coloneqq \ord(\delta_{Z(G)})$. Then
\begin{enumerate}
\item $d \leq p^b-1$,
\item $\delta^d$ induces an automorphism of $G/Z(G)$ whose order $c$ is the order of a $p'$-element of $\SP_{2a} (p)$,
\item $\ord(\delta) = cd$, where $c$ is as in (2).
\end{enumerate}
\end{lemma}

\begin{proof}
As $G' = Z(G)$ is characteristic, we have $\delta_{Z(G)} \in \Aut (Z(G))$. Since $Z(G)$ is elementary abelian of order $p^b$, we know that $\Aut (Z(G))$ is isomorphic to $\GL_b (p)$. Thus, $d \leq \max(\GL_b(q)) = p^b-1$.

We claim that $\delta^d$ induces an element $\delta'$ of $\SP_{2a} (p)$ of the same order. (For the details of this argument, see Section 4 of \cite{brown}.)  Note that $\delta^d$ induces an automorphism $\delta'$ of $G/Z(G)$.  Since $(\delta_{Z(G)})^d = 1$, we see that $\delta'$ has the same order as~$\delta^d$.  Also, $\delta^d$ stabilizes all of the irreducible characters of $G' = Z(G)$.  Hence, $\delta^d$ preserves the bilinear form $\langle \cdot, \cdot \rangle_\phi$ for every irreducible character $\phi \in \Irr(Z(G))$.  We now fix a nontrivial irreducible character $\phi \in \Irr ({Z(G)})$.  Having $\delta^d$ act on $G/Z(G)$ preserving $\langle \cdot, \cdot \rangle_\phi$, we see that $\delta'$ is an element of $\SP_{2a} (p)$. In particular, $c \coloneqq \ord(\delta') = \ord (\delta^d)$ is the order of a $p'$-element in $\SP_{2a}(p)$. 
We deduce that $\delta^{dc} = (\delta^d)^c = 1$.  Note that if $\ord (\delta) < cd$, then either $\delta_{Z(G)}$ has order less than $d$ or $\delta'$ has order less than $c$; therefore, $\ord (\delta) = cd$.
\end{proof}




We next consider the orders  of $p'$-elements in $\SP_{2a}(p)$. 
For integers $a, m > 1$, a  \emph{Zsigmondy prime divisor} of $a^m-1$ is a prime $l$ that divides $a^m-1$, but not $a^i-1$ for $i \in \{1, \dots, m-1\}$.   Note that since $a^{2m} - 1 = (a^m - 1)(a^m + 1)$, the Zsigmondy prime divisors of $a^{2m} - 1$ are all divisors of $a^m + 1$.


\begin{lemma}\label{max SP}
Let $G =\SP_{2a}(p)$. 
\begin{enumerate}
\item If $a = 1$ the maximal order of a $p'$-element of $G$ is $p + 1$.
\item If $(p,a) = (2,3)$ the maximal order of a $p'$-element ($2'$-element) in $G$ that
divides $p^{2a} - 1 = 63$ is $p^a + 1 = 9$. 
\item If $a > 1$ or if $(p,a) \neq (2,3)$, 
$k$ is the order of a $p'$-element in $G$, and $k$ is divisible by all Zsigmondy prime divisors of $p^{2a} - 1$, then $k$ divides $p^a + 1$.
\end{enumerate} 
\end{lemma}

\begin{proof}
When $a = 1$, we know that $\SP_2 (p) \cong \SL_2 (p)$ and by Dickson's Classification of the subgroups of ${\rm PSL}_2 (q)$, the maximal order of a $p'$-element is $p+1$. If $(p,a) = (2,3)$, we can compute the orders of the $2'$-elements of $\SP_6 (2)$ in Magma (or GAP). These are $\{ 1, 3, 5, 7, 9, 15 \}$. Thus, $9$ is the maximal $2'$-order which divides $63$, as required.

Now let $a> 1$ and $(p,a) \neq (2,3)$, so the Zsigmondy prime theorem applies. Since $p^{2a} - 1 = (p^a -1) (p^a+1)$, we see that $p^a + 1$ is divisible by the Zsigmondy prime divisors of $p^{2a} - 1$.  Suppose conversely that $k$ is the order of a $p'$-element of $G$, and $k$ is divisible by the Zsigmondy prime divisors of $p^{2a} -1$.   The $p'$-elements of $G$ are semi-simple and thus lie in maximal tori of $G$. These are of the form $\prod_{i=1}^l Z_i$ for $l \in \N$, where $Z_i$ is a cyclic group of order $p^{j_i} \pm 1$, and $j_1 + \dots + j_l = a$ (see, for instance, \cite[Section~2]{KAN}).
In particular, $k$ divides ${\rm lcm}_{i=1}^n (p^{u_i} - 1) \,{\rm lcm}_{j=1}^m (p^{v_j} + 1)$ for some $u_1, \dots, u_n, v_1, \dots, v_m, n,m \in \Z_{\geq 0}$ with $\sum_{i=1}^n u_i + \sum_{j=1}^m v_j \le a$. Let $q$ be a Zsigmondy prime divisor of $p^{2a} - 1$.  Then $q$ must divide some $p^{u_i} - 1$ or $p^{v_j} + 1$.  Due to $u_i \leq a$, the first case cannot occur, so $q$ divides $p^{v_j} + 1$ for some~$j$.  If $v_j < a$, then $q$ divides $p^{2v_j} - 1 < p^{2a} - 1$, which is a contradiction.  Hence, we deduce that $v_j = a$.  Since $\sum_{i=1}^n u_i + \sum_{j=1}^m v_j \leq a$, this implies that $k$ divides $p^a + 1$. 
\end{proof}

We now show that when $G$ has the given automorphism of order $|G:Z(G)| - 1 = |G:G'| - 1$, then $G$ is ultraspecial, and we can determine the order of the restriction of the automorphism to the center.  

\begin{corollary} \label{ultraspecial}
Let $G$ be a semi-extraspecial $p$-group with $|G:Z(G)| = p^{2a}$ and $|Z(G)| = p^b$ having an automorphism $\sigma$ of order $p^{2a} - 1$.  Then the following hold:
\begin{enumerate}
\item $G$ is ultraspecial.
\item $\sigma_{Z(G)}$ has order $|Z(G)| - 1 = p^a - 1$.
\item $\sigma^{p^a - 1}$ induces an automorphism of $G/Z(G)$ having order $p^a + 1$.
\end{enumerate}
\end{corollary}

\begin{proof}
Write $|G:Z(G)| = p^{2a}$ and $|Z(G)| = p^b$ for some $a,b \in \N$. Proving that $G$ is ultraspecial is equivalent to showing $a = b$.  Let $d$ be the order of $\sigma_{Z(G)}$ and let $c$ be the order of the automorphism of $G/Z(G)$ induced by $\sigma^d$.  By Lemma \ref{max auto} (3), we have $p^{2a} - 1 = \ord (\sigma) = cd$.  We know by Lemma \ref{max auto} (1) that $d \le p^b - 1 \le p^a -1 $.  Notice that $c = (p^{2a}- 1)/d \ge (p^{2a} - 1)/(p^b - 1) \ge (p^{2a} - 1)/(p^a - 1) = p^a + 1$. Also, $c$ is the order of a $p'$-element of $\SP_{2a}(p)$ by Lemma \ref{max auto}(2).  When $a = 1$ or $(p,a) = (2,3)$, it follows that $c = p^a + 1$ by Lemma \ref{max SP}. 

Now assume $a > 1$ and $(p,a) \neq (2,3)$.   As~$d$ is the order of a $p'$-element of $\GL_b (p)$, it follows that $d$ divides $\Pi_{j=1}^{b} (p^{j}-1)$ since $|\GL_b(p)| = \Pi_{i=0}^{b-1} (p^b - p^i)$.  Thus, no Zsigmondy prime of $p^{2a} - 1$ divides $d$, and so, these primes must divide $c$. Hence, $c$ divides $p^a + 1$ by Lemma \ref{max SP}.  This implies that $c \le p^a + 1$.  In the previous paragraph, we showed that $p^a + 1 \le c$.   
Thus, in all cases $c = p^a +1$.  It follows that $d = p^a -1$, so $a=b$. 
Conclusions (2) and (3) follow.
\end{proof}

For every semi-extraspecial $p$-group $G$, every automorphism of $G$ induces a linear map on the elementary abelian group $G/G' = G/Z(G)$.  We say that $\alpha \in \Aut(G)$ acts \emph{irreducibly} on $G/Z(G)$ if $\alpha$ does not preserve any nontrivial $\F_p$-subspace of~$G/Z(G)$ (that is, there is no subspace $0 \lneq W \lneq G/Z(G)$ with $\alpha(W) \subseteq W$). 
The key ingredient of the proof of Theorem~\ref{Intro} is the following result:

\begin{theorem}[{\cite[Theorem~2]{BEI79}}]\label{theorem:beisiegel}
Let $\psi$ be a $p'$-automorphism of the non-abelian special finite $p$-group $G$ acting irreducibly on $G/G' = G/Z(G)$ and trivially on $G' = Z(G)$. Set $$m \coloneqq \min \{k \in \mathbb{N} \colon \ord(\psi) \mid (p^k-1)\}.$$ Then $G$ is isomorphic to $Q/Z$, where $Q$ is a Sylow $p$-subgroup of $\SU_3(p^{m})$ and $Z$ is a subgroup of $Z(Q)$.
\end{theorem}

We can now prove the first implication of Theorem~\ref{Intro}. 

\begin{theorem}\label{theo:classification}
Let $G$ be a semi-extraspecial $p$-group and suppose that~$G$ has an automorphism of order $|G:G'|-1$. Then~$G$ is isomorphic to a Sylow $p$-subgroup of~$\SU_3(p^{2a})$. 
\end{theorem}

\begin{proof}
Let $\sigma$ be an automorphism of $G$ of order $|G:G'|-1$ and let $\psi \coloneqq \sigma^{p^a-1}$. We show that $\psi$ satisfies the prerequisites of Theorem~\ref{theorem:beisiegel}. By Corollary~\ref{ultraspecial}\,(2), $\sigma_{Z(G)}$ has order $p^a-1$, so $\psi$ acts trivially on $Z(G)$. We claim that 
\[
\min\{k \in \N \colon \ord(\psi) \mid p^k-1\}= 2a.
\]
To see this, let $k \in \N$ be minimal such that $\ord(\psi) = p^a +1$ divides $p^k-1$. Clearly, we have $k \leq 2a$. Now write $m \cdot (p^a+1) = p^k-1$. Then $m \equiv -1 \pmod{p^a}$ follows. In particular, we obtain $m \geq p^a-1$ and hence $k = 2a$. Finally, it follows from \cite[Theorems~2.3.2 and~2.3.3]{Sho} that the action of $\psi$ on $G/Z(G)$ is irreducible. 
	
With this, Theorem~\ref{theorem:beisiegel} yields that $G$ is isomorphic to $Q/Z$ for a Sylow $p$-subgroup $Q$ of $\SU_3(p^{2a})$ and a subgroup~$Z$ of $Z(Q)$. Due to $|G| = p^{3a} = |Q|$ (this follows from Corollary~\ref{ultraspecial}\,(1)), we obtain $G \cong Q$ and the claim follows. 
\end{proof}

\begin{Remark}\label{rem:slsu}
For $p > 2$, the Sylow $p$-subgroups of $\SU_3(p^{2k})$ and $\SL_3(p^k)$ are isomorphic. To see this, observe that a Sylow $p$-subgroup $P$ of $\SU_3 (p^{2k})$ is ultraspecial of order $p^{3a}$ and all the centralizers of noncentral elements of $P$ are abelian. By \cite[Theorem~5.10]{Ver}, $P$ is isomorphic to a Sylow $p$-subgroup of $\SL_3 (p^k)$. A colleague has pointed out that this can also be proven directly using unitriangular matrices.

In contrast, one can show that the Sylow $2$-subgroups of $\SL_3 (2^k)$ and $\SU_3(2^{2k})$ contain different numbers of involutions and hence they are not isomorphic.
\end{Remark}

Remark~\ref{rem:slsu} shows that Corollary~\ref{Intro odd} is an immediate consequence of Theorem~\ref{Intro}. As a converse to Theorem~\ref{theo:classification}, we note the following: 

\begin{lemma}\label{lemma:suauto}
Let $P$ be a Sylow $p$-subgroup of $\SU_3(p^{2k})$. Then $P$ has an automorphism of order~$p^{2k}-1$. 
\end{lemma}

\begin{proof}
This is essentially proved as Satz II.10.12 (b) in \cite{HUP67}.  
\end{proof}

This completes the proof of Theorem~\ref{Intro}. For $p> 2$, the automorphism groups of the Sylow $p$-subgroups of $\SL_3 (p^k)$ have been constructed in \cite[Theorem~4.1]{LeWi}. This can be used for an alternative proof of Lemma~\ref{lemma:suauto} if $p$ is odd. 

We end the note with an observation by Bill Kantor from a presentation of this paper at the conference in Denver Colorado for the 75th birthday of Jon Hall in August 2024.  As we have noted above, our automorphism $\sigma$ acts transitively on $Z(P) \setminus \{1\}$.  On the other hand, we know that $\sigma$ transitively permutes the nontrivial cosets of $Z(P)$ in $P$ and each element in $P \setminus Z(P)$ is conjugate to its coset by $Z(P)$.  Hence, all the elements outside of $Z(P)$ have the same order.  When $P$ is a $2$-group, this implies that $Z(P)$ contains all of the involutions and $\sigma$ is transitively permuting the involutions.  It follows that $P$ is a Suzuki $2$-group.  

In \cite{Higman}, Higman classified the Suzuki $2$-groups into four types: type A, type B, type C, and type D. The groups of type A have been studied more widely (see Section VIII.6 in \cite{HuBlII}).  On the other hand, the groups of type B, type C, and type~D are often S.E.S. groups, whereas the groups of type A are not (see Table 3, Table 4, and  Lemma 2.5 in \cite{DFH}).  We have seen that the groups in Theorem \ref{Intro} are isomorphic to the Sylow $2$-subgroups of $\SU_3 (2^{2a})$.  The fact that these groups are type B is mentioned in the second paragraph of \cite{collins}.  We also refer the reader to the argument given in \cite{overflow}.


\section*{Acknowledgments}

The research of the first author has received funding from the European Research Council (ERC) under the European Union’s Horizon 2020 research and innovation programme (EngageS: grant agreement No.~820148) and from the German Research Foundation DFG (SFB-TRR 195 “Symbolic Tools in Mathematics and their Application”). We thank  Joshua Maglione, Geoffrey Robinson, and David Craven for some helpful comments. 

\end{document}